\documentclass[letterpaper,11pt]{article}

\usepackage{amsmath,amsfonts,amssymb,amsthm,mathrsfs,dsfont}
\usepackage{cases}
\usepackage{enumerate}
\usepackage[usenames,dvipsnames]{color}
\usepackage{verbatim} 
\usepackage{array} 
\usepackage{graphicx,epstopdf} 
\usepackage{mathtools}

\newcommand{\E}[0]{\mathbb{E}}

\newtheorem{thm}{Theorem}[section]

\newtheorem{lemma}[thm]{Lemma}

\newtheorem{prob}[thm]{Problem}

\newcommand{\beq}[1]{\begin{equation}\label{#1}}
\newcommand{\enq}[0]{\end{equation}}

\newcommand{\bn}[0]{\bigskip\noindent}
\newcommand{\mn}[0]{\medskip\noindent}
\newcommand{\nin}[0]{\noindent}

\newcommand{\sub}[0]{\subseteq}
\newcommand{\sm}[0]{\setminus}

\newcommand{\eee}[0]{{\cal E}}

\newcommand{\h}[0]{{\cal H}}

\newcommand{\ra}[0]{\rightarrow}

\newcommand{\cc}[0]{\gamma}

\newcommand{\0}[0]{\emptyset}

\newcommand{\C}[2]{{{#1}\choose{{#2}}}}
\newcommand{\Cc}[0]{\tbinom}
\newcommand{\ga}[0]{\alpha }
\newcommand{\gb}[0]{\beta }

\newcommand{\gd}[0]{\delta }

\newcommand{\gl}[0]{\lambda }
\newcommand{\gL}[0]{\Lambda}
\newcommand{\go}[0]{\omega}
\newcommand{\gO}[0]{\Omega}

\newcommand{\gs}[0]{\sigma}

\newcommand{\gz}[0]{\zeta}
\newcommand{\eps}[0]{\varepsilon }

\newcommand{\vp}[0]{\varphi}

\newcommand{\ft}[0]{s}
\newcommand{\fP}[0]{f}

\newcommand{\comments}[1]{}
\usepackage[usenames]{color}

\newcommand{\pr}[1][]{\mathbb{P}}

\newcommand{\Ee}{\mathcal{E}}

\newcommand{\Aitch}{\mathcal{H}}

\newcommand{\MPOut}[3]{\kout{\mpart{#1}{#3}}{#2}}
\newcommand{\kout}[2]{{#1}(#2\text{-out})}
\newcommand{\complete}[2]{K_{#1} ^{(#2)}}
\newcommand{\mpart}[2]{K_{[{#1}] ^{#2}}}

\newcommand{\lcho}[2]{{#1 \choose #2}}

\newcommand{\1}{\textbf{1}}
\newcommand{\ignore}[1]{}

\newcommand{\dd}[1]{\,\textit{d}#1}

\newtheorem{theorem}{Theorem}
\newtheorem{proposition}[theorem]{Proposition}
\newtheorem{corollary}[theorem]{Corollary}

\newtheorem{conjecture}{Conjecture}

\begin{document}
\renewcommand{\thefootnote}{\fnsymbol{footnote}}

\title{Perfect fractional matchings in $k$-out hypergraphs\footnote{Supported by NSF grant DMS1501962.}}
\author{Pat Devlin\footnote{Department of Mathematics, Rutgers University, Piscataway, NJ 08854 (\texttt{prd41@math.rutgers.edu}, \texttt{jkahn@math.rutgers.edu})}, Jeff Kahn\footnotemark[2]}
\date{March 7, 2017}

\maketitle
\renewcommand*{\thefootnote}{\arabic{footnote}}

\begin{abstract}
Extending the notion of (random) $k$-out graphs, we consider when the $k$-out hypergraph is likely to have a
perfect fractional matching.  In particular, we show that for each $r$ there is a $k=k(r)$
such that the $k$-out $r$-uniform hypergraph on $n$ vertices
has a perfect fractional matching with high
 probability (i.e., with probability tending to $1$ as $n\ra\infty$)
 and prove an analogous result for $r$-uniform $r$-partite hypergraphs.  This is based on
 a new notion of hypergraph expansion and the observation that sufficiently expansive hypergraphs
 admit perfect fractional matchings.  As a further application, we give
 a short proof of a stopping-time result originally due to Krivelevich.
\end{abstract}

\section{Introduction}
Hypergraphs constitute a far-reaching
generalization of graphs and a basic combinatorial construct but are notoriously difficult to work with.
A \textit{hypergraph}
is a collection $\Aitch$ of subsets
(\emph{``edges"}) of a set $V$ of \emph{``vertices."}
Such an $\Aitch$ is $r$-\emph{uniform} (or an \emph{r-graph}) if
each edge has cardinality $r$ (so $2$-graphs are graphs).
A \textit{perfect matching} in a hypergraph is a collection of edges partitioning the vertex set.
For any $r > 2$, deciding whether an $r$-graph has a perfect matching is an NP-complete problem \cite{karp};
so instances of the problem tend to be both interesting and difficult.
Of particular interest here has been trying to understand conditions under which a \emph{random} hypergraph
is likely to have a perfect matching.

\mn

The most natural model of a random $r$-graph is the ``Erd\H{o}s-R\'enyi" model, in which
each $r$-set is included in $\Aitch$ with probability $p$, independent of other choices.
One is then interested in the ``threshold," roughly, the order of magnitude of $p=p_r(n)$ required
to make a perfect matching likely.
Here the graph case was settled by Erd\H{o}s and R\'enyi
\cite{erdos1964random, erd1966}, but for $r > 2$ the problem---which became known as
Shamir's Problem following \cite{erdos1981combinatorial}---remained open until \cite{johansson2008factors}.
In each case, the obvious
obstruction to containing a perfect matching is existence of an isolated vertex (that is,
a vertex contained in no edges), and a natural guess is that this is the \emph{main} obstruction.
A literal form of this assertion---the \emph{stopping time} version---says that if we choose random
edges \emph{sequentially}, each uniform from those as yet unchosen, then we w.h.p.\footnote{As usual we use \emph{with high probability (w.h.p.)}\ to
mean with probability tending to 1 as the relevant parameter---here always $n$---tends to infinity.}\
have a perfect matching as soon as all vertices are covered.
This nice behavior does hold for graphs \cite{bollobas1985random}, but for hypergraphs remains conjectural
(though at least the value it suggests for the threshold is correct).

\mn

An interesting point here is that taking $p$ large enough to avoid isolated vertices produces many more edges
than other considerations---e.g., wanting a large \emph{expected} number of perfect matchings---suggest.
This has been one motivation for the substantial body of work on models of
random graphs in which isolated vertices are automatically avoided, notably
random \emph{regular} graphs (e.g., \cite{wormald1999models})
and the $k$-out model.
The generalization of the latter to hypergraphs, which we now introduce,
will be our main focus here.

\begin{quote}
\textbf{The $k$-out model.}
For a (``host") hypergraph $\h$ on $V$,
$\kout{\h}{k}$ is the random subhypergraph $\cup_{v\in V}E_v$, where
$E_v$ is chosen uniformly from the $k$-subsets of $\h_v:=\{A\in \h:v\in A\}$
(or---but we won't see this---$E_v =\h_v$ if $|\h_v|<k$),
these choices made independently.
\end{quote}

\mn

The $k$-out model for $\h=K_{n,n}$ (the complete bipartite graph) was introduced by
Walkup \cite{walkup1980matchings}, who
showed that w.h.p.\ $\kout{K_{n,n}}{2}$ is
Hamiltonian, so in particular contains a perfect matching, and Frieze \cite{frieze1986maximum}
proved the nonbipartite counterpart of the matching result,
showing that $\kout{K_{2n}}{2}$ has a perfect matching w.h.p.\  (Hamiltonicity in the latter case
turned out to be more challenging; it
was studied in \cite{fenner1984hamiltonian, frieze1987hamiltonian, cooper1994hamilton}
and finally resolved by Bohman and Frieze \cite{bohman2009hamiltonian}, who proved $\kout{K_n}{3}$ is
Hamiltonian w.h.p.).  The idea of a general host $G$ was introduced by Frieze and T.~Johansson
\cite{frieze2014kOut}; see also e.g.,
Ferber \emph{et al.} \cite{ferber2014packing} for (\emph{inter alia}) a nice connection with $G_{n,p}$.

For \emph{hypergraphs} the $k$-out model seems not to have been studied previously (random regular hypergraphs \emph{have} been considered,
e.g., in \cite{cooper1996perfect}).
Here the two most important examples would seem to be
$\Aitch = \complete{n}{r}$ (the complete $r$-graph on $n$ vertices) and $\Aitch = \mpart{n}{r}$ (the complete $r$-partite $r$-graph with $n$ vertices in each part).
It is natural to expect that for each of these
there is some $k = k(r)$ for which $\kout{\Aitch}{k}$ has a perfect matching w.h.p..
Note that, while almost certainly correct, these are likely to be difficult, as
either would
imply the aforementioned resolution of Shamir's Problem;
still, we would like to regard the following
linear relaxations as a small step in this direction.
(Relevant definitions are recalled in Section \ref{sectionPrelims}.)

\begin{theorem}\label{generalHypergraph}
For each $r$, there is a k such that w.h.p.\
$\kout{\complete{n}{r}}{k}$ admits a perfect fractional matching
and $w\equiv 1/r$ is the only fractional cover of weight $n/r$.
\end{theorem}

\begin{theorem}\label{r-partiteHypergraph}
For each $r$, there is a k such that w.h.p.\
$\h=\kout{\mpart{n}{r}}{k}$
admits a perfect fractional matching and each minimum weight
fractional cover of $\h$ is constant on each block of the r-partition.
\end{theorem}

Our upper bounds on the $k$'s are quite large (roughly $r^{r}$), but in fact we don't even know that
they must be larger than $2$ (though this sounds optimistic),
and we make no attempt to optimize.
In the more interesting case of (ordinary) perfect matchings,
consideration of the expected number of perfect matchings shows that
$k$ does need to be be at least exponential in $r$.

\mn

We will make substantial use of the next observation (or, in the $r$-partite case,
of the analogous  Proposition~\ref{r-partite-expansion}, whose statement we postpone), in which the
notion of expansion may be of some interest.
Recall that an \emph{independent set} in a hypergraph is a set of vertices
containing no edges.

\begin{proposition}\label{mainLemma}
Suppose $\Aitch$ is an $r$-graph in which, for all disjoint $X, Y \subseteq V$ with
$X$ independent and
\beq{Y}
|Y| < (r-1)|X|,
\enq
there is some edge meeting $X$ but not $Y$.
Then $\Aitch$ has a perfect fractional matching.
If, moreover we replace ``$<$" by ``$\leq$" in \eqref{Y},
then
$w\equiv 1/r$ is the only fractional cover of weight $n/r$.
\end{proposition}
\noindent
It's not hard to see that for $r>2$ the proof of this can be tweaked to
give the stronger conclusion even under the weaker hypothesis.  (For $r=2$ this is
clearly false, e.g., if $G$ is a matching.)

Related notions of expansion (respectively stronger than and incomparable to ours)
appear in \cite{krivelevich} and \cite{haber}.
An additional application of Proposition~\ref{mainLemma},
given in Section~\ref{theorem Krivelevich Stopping}, is a
short alternate proof of the
following result of Krivelevich \cite{krivelevich}.
\begin{theorem}\label{theorem Krivelevich Stopping}
Let $\{\Aitch_{t}\}_{t\geq 0}$ denote the random $r$-graph process on $V$
in which each step adds an edge chosen uniformly from the current non-edges,
let $T$ denote the first $t$ for which $\Aitch_{t}$ has no isolated vertices.
Then $\Aitch_{T}$ has a perfect fractional matching w.h.p..
\end{theorem}

\mn
\textbf{Outline.}
Section~\ref{sectionPrelims} includes definitions and brief linear programming background.
Section~\ref{sectionGeneralHypergraph} treats $\complete{n}{r}$, proving
Proposition~\ref{mainLemma} and Theorem~\ref{generalHypergraph},
and the corresponding results for $\mpart{n}{r}$ are proved in
Section~\ref{sectionR-partite}.
Finally, Section~\ref{sectionStopping}
returns to $\complete{n}{r}$,
using Proposition~\ref{mainLemma}
to give an alternate proof of Theorem~\ref{theorem Krivelevich Stopping}.

\section{Preliminaries}\label{sectionPrelims}
Except where otherwise specified, $\h $ is an $r$-graph on $V=[n]$.
As usual, we use $[t]$ for $\{1, 2, \ldots , t\}$ and $\lcho{X}{t}$ for
the collection of $t$-element subsets of $X$.
Throughout we use $\log$ for $\ln$ and take asymptotics as $n\ra \infty$
(with other parameters fixed), pretending (following a common abuse) that all large numbers
are integers and assuming $n$ is large enough to support our arguments.

\mn

We need to recall a minimal amount of linear programming background (see e.g., \cite{schrijver1998}
for a more serious discussion).
For a hypergraph $\Aitch$,
a \textit{fractional (vertex) cover} is a map $w : V \to [0,1]$ such that
$\sum_{v \in e} w(v) \geq 1$ for all $e \in \Aitch$; the \textit{weight} of a cover $w$
is $|w| = \sum_{v} w(v)$; and the \emph{fractional cover number}, $\tau^*(\h)$, is the largest
such weight.
Similarly a \textit{fractional matching} of $\h$ is a
$\vp  : \h \to [0,1]$ such that $\sum_{e \ni v} \vp (e) \leq 1$ for all $v \in V$; the
weight of such a $\vp$ is defined as for fractional covers;
and the \emph{fractional matching number}, $\nu^*(\h)$, is the \emph{maximum} weight of a fractional matching.

In this context, LP-duality says that
$\nu^{\ast}(\Aitch) = \tau^{\ast}(\Aitch)$ for any hypergraph.
For $r$-graphs the common value is trivially at most $n/r$ (e.g., since $w\equiv 1/r$ is a fractional cover).
A fractional matching in an $r$-graph is \emph{perfect} if it achieves this bound; that is,
if $\sum\vp_e=n/r$
(equivalently $\sum_{e\ni v}\vp_e=1 ~\forall v$, which would be the definition of perfection
in a nonuniform $\h$).

\mn

Finally, given $\h$ we say a nonempty $X\sub V$ is
\emph{$\lambda$-expansive} if for all $Y\sub V\sm X$
of size at most $\lambda |X|$, there is some edge meeting $X$ but not $Y$.

\section{Proofs of Proposition of \ref{mainLemma} and Theorem \ref{generalHypergraph}}\label{sectionGeneralHypergraph}
\begin{proof}[Proof of Proposition~\ref{mainLemma}]
It is enough to show that if $w$ is a fractional cover with
$t_0:= 1/r-\min_v w(v)>0$, then $|w|\geq n/r$, with the inequality strict if we assume
the stronger version of \eqref{Y}.
We give the argument under this stronger assumption; for the weaker, just
replace the few strict inequalities below by nonstrict ones.
Given $w$ as above, set, for each $t> 0$,
\[
\mbox{$W_{t} = \{v \in [n] \ : \ w(v) \leq \frac{1}{r} - t\}$,
$~~~W^{t} = \{v \in [n] \ : \ w(v) \geq \frac{1}{r} + t\}$.}
\]
Since $w$ is a fractional cover,
each edge meeting $W_t$ must also meet $W^{t/(r-1)}$
(or the weight on the edge would be less than 1);
so, since $W_t$ is independent, the hypothesis of Proposition~\ref{mainLemma}
gives $|W^{t/(r-1)}| > (r-1) |W_t|$
for $t\in (0,t_0]$ (the $t$'s for which $W_t\neq \emptyset$).

\mn

For $\ft \in \mathbb{R}$, define $\fP(\ft) = |\{v \in [n] \ : \ w(v) \geq \ft\}|$.
Then
\begin{eqnarray*}
\mbox{$\displaystyle \int_{0}^{1} \fP(\ft) \dd{\ft}$}
&=&
\mbox{$\displaystyle \int_{0} ^{1} \sum_{v \in [n]} \1_{\{w(v)\geq \ft\}}\dd{\ft}$}\\
&=& \mbox{$\displaystyle \sum_{v \in [n]} \int_{0} ^{1} \1_{\{w(v)\geq \ft\}} \dd{\ft}
= \sum_{v \in [n]} w(v)  = \tau^{\ast} ( \Aitch).$}
\end{eqnarray*}
We also have $|W^t| = \fP(1/r + t)$ and
$|W_t| \geq n-\fP(1/r - t)$,
implying
\[
\fP(1/r + t/(r-1))\geq (r-1) (n-\fP(1/r -t)),
\]
with the inequality strict if $t\in (0,t_0]$.  Thus,
\begin{eqnarray*}
\tau^{\ast} (\Aitch) &=& \int_{0} ^{1} \fP(\ft) \dd{\ft}
= \int_{0} ^{1/r} \fP(\ft) \dd{\ft} + \int_{1/r} ^{1} \fP(\ft) \dd{\ft}\\
&=& \int_{0} ^{1/r} \fP(1/r - t) \dd{t} + \int_{0} ^{(r-1)^2 /r} \dfrac{\fP(1/r + t/(r-1))}{r-1} \dd{t}\\
&\geq& \int_{0} ^{1/r} \left[\fP(1/r - t) + \dfrac{\fP(1/r + t/(r-1))}{r-1}\right] \dd{t}\\
&>& \int_{0} ^{1/r} \left[\fP(1/r - t) + (r-1)\dfrac{n - \fP(1/r -t)}{r-1} \right]\dd{t} = \dfrac{n}{r}.
\end{eqnarray*}
\end{proof}

We should perhaps note that the converse of Proposition~\ref{mainLemma} is not true in general
(failing, e.g., if $r>2$ and $\h$ is itself a perfect matching). But
in the graphic case ($r=2$) the converse \emph{is} true (and trivial), and the proposition
provides an alternate proof of the following characterization, which is \cite[Thm.\ 2.2.4]{scheinerman2011} (and is also contained in \cite[Thm.\ 2.1]{alon2007independent}, e.g.).
\begin{corollary}
A graph has a perfect fractional matching iff $|N(I)| \geq |I|$ for all independent $I$.
\end{corollary}
\nin
(where $N(I) $ is the set of vertices with at least one neighbor in $I$).

\begin{proof}[Proof of Theorem~\ref{generalHypergraph}]
Given $r$, let (without trying to optimize) $k=(2r^2)^r$ and $c=k^{-1/r}=1/(2r^2)$,
and let $\Aitch =\kout{\complete{n}{r}}{k}$.
Theorem~\ref{generalHypergraph} (with this $k$) is
an immediate consequence of Proposition~\ref{mainLemma} and the next
two routine lemmas.  (As usual $\alpha(\Aitch)$ is
the size of a largest independent set in $\h$.)

\begin{lemma}\label{outInd}
W.h.p. $\alpha(\Aitch) < cn$.
\end{lemma}
\begin{lemma}\label{outExpand}
W.h.p.\ every $X \subseteq V(\Aitch)$ with $|X| \leq c n$ is $(r-1)$-expansive.
\end{lemma}

\begin{proof}[Proof of Lemma~\ref{outInd}.]
The probability that $S\in \C{[n]}{s}$ is independent in $\h$ is
\[
\left[1 - \tfrac{(s-1)_{r-1}}{(n-1)_{r-1}}\right] ^{sk} < \exp\left[ - sk\left(\tfrac{s-r}{n} \right)^{r-1}\right].
\]
(where $(a)_b = a(a-1)\cdots (a-b+1)$),
and summing this over $S$ of size $cn$ bounds $\pr (\alpha \geq cn) $ by
\[
2^n \exp\left[ - cn k (c-r/n )^{r-1}\right]
= \exp\left[n \left( \ln 2 - (1-o(1)) k c^r \right) \right],
\]
which tends to 0 as desired.
\end{proof}

\begin{proof}[Proof of Lemma~\ref{outExpand}.]
For $X$, $Y$ disjoint subsets of $[n]$,
let $B(X,Y)$ be the event that $Y$ meets all edges meeting $X$.  Then,
with $x=|X|$ and $y=|Y|$,
\[
\pr (B(X,Y)) \leq \left[1-\tfrac{(n-y-1)_{r-1}}{(n-1)_{r-1}} \right]^{kx} \leq \left[1-\left(\tfrac{n-y-r}{n}\right)^{r-1} \right]^{kx} \leq \left[\tfrac{r(y+r)}{n} \right]^{kx},
\]
the last inequality following from
\beq{mx}
1-(1-x)^m \leq mx
\enq
(valid for $x\in [0,1] $ and nonnegative integer $m$).
The probability that the conclusion of the lemma fails is thus less than
\begin{eqnarray*}
\mbox{$\displaystyle \sum  \Cc{n}{rx} \Cc{rx}{x} \left[\tfrac{r(y+r)}{n} \right]^{kx}$}
&< &
\mbox{$\displaystyle \sum \left(\tfrac{ne}{rx} \right)^{rx} 2^{rx} \left[\tfrac{r(y+r)}{n} \right]^{kx}$}\\
&=& \mbox{$\displaystyle \sum\left[ (2e)^r  \left(\tfrac{rx}{n} \right)^{k-r} ((r-1)+r/x)^k \right]^{x} $}\\
&< & \mbox{$\displaystyle \sum \left[ (4er)^r  (r(2r-1)x/n)^{k-r} \right]^{x}$} ~=o(1),
\end{eqnarray*}
where the sums are over $1\leq x\leq cn$.
\end{proof}

\end{proof}

\section{Proof of Theorem~\ref{r-partiteHypergraph}}\label{sectionR-partite}
As in the proof of Theorem~\ref{generalHypergraph} we first show that the conclusions of
Theorem~\ref{r-partiteHypergraph} are implied (deterministically) by sufficiently good expansion
and then show that $\MPOut{n}{k}{r}$ w.h.p.\ expands as desired.
We take $V = V_1 \cup \cdots \cup V_r$ to be our $r$-partition
(so $|V_i|=n$\, $\forall i$) and below always assume $\h \sub \mpart{n}{r}$.
\begin{proposition}\label{r-partite-expansion}
Suppose $\eps\in (0,1/2)$ and $\lambda> 2r^2$ are fixed and $\Aitch $
satisfies: for any $i\in [r]$, $T\sub V_i$, $U_j\sub V_j$ for $j\neq i$ and $U=\cup_{j\neq i}U_j$,
there is an edge meeting $T$ but not $U$ provided either
\begin{itemize}
\item[(i)] $|T| \leq \varepsilon n$ and $|U_j| \leq \lambda |T|$ $\forall j\neq i$, or
\item[(ii)] $|T| \geq \eps n$ and $|U_j| \leq (1- \varepsilon) n$ $\forall j\neq i$.
\end{itemize}
Then $\Aitch$ admits a perfect fractional matching, and every minimum weight fractional
cover of $\Aitch$ is constant on each $V_i$.
\end{proposition}
\begin{proof}
Define a \textit{balanced assignment} to be a $w : V \to \mathbb{R}$ with
$\sum_{v \in V_i} w(v) = 0$ and $w(e) \geq 0$ for all $e \in \Aitch$.

We claim that (under our hypotheses) the only balanced assignment is the trivial $w \equiv 0$.
To get Proposition \ref{r-partite-expansion} from this, let $f$ be a minimum weight fractional cover, and let $w_f (v) = f(v) - \sum_{u \in V_i} f(u) / n$, for each $i$ and $v \in V_i$.
Then $w_f$ is a balanced assignment:  $\sum_{v \in V_i} w_f (v) = 0$ is obvious and nonnegativity
holds since $f(e) \geq 1$ and, by minimality, $\sum_{v \in V} f(v) \leq n $.  Thus $w_f \equiv 0$, implying $f$ is as promised.

\mn

Suppose then that $w $ is a balanced assignment.  For $X \subseteq V$ and $t \geq 0$, set $X^t = \{v \in X : w(v) \geq t\}$, $X_t = \{v \in X : w(v) < -t\}$, $X^{+} = X^0$ and $X^{-} = X_0$, and define the \textit{value} of $X$ to be $\psi (X) = \sum_{v \in X} |w(v)|$.  Let $S=\{i\in [r]: |V_i ^{-}| \leq \varepsilon n\}$ and $B = [r] \setminus S$.

\begin{lemma}\label{small sets have no value}
If $X \subseteq V ^{-}$ and $|X| \leq \varepsilon n$, then $\psi (X) \leq r\psi (V^{+}) /\lambda$.
\end{lemma}
\begin{proof}
For any $t > 0$, note that every edge meeting $X_{t}$ meets $V^{t/(r-1)}$ since otherwise,
we could find an edge of negative weight.  So since $|X_t| \leq |X| \leq \varepsilon n$,
condition (i) implies $|V^{t/(r-1)}| \geq \lambda |X_t|$.  Thus,
\begin{eqnarray*}
\psi (V^{+}) &=& \int_{0} ^{\infty} |V^{u}| \dd{u} =
\dfrac{1}{r-1}\int_{0} ^{\infty} |V^{t/(r-1)}| \dd{t}\\
&\geq & \dfrac{\lambda}{r-1}\int_{0} ^{\infty} |X_{t}| \dd{t} = \dfrac{\lambda}{r-1} \psi (X).\qedhere
\end{eqnarray*}
\end{proof}
\begin{lemma}\label{large sets have no value}
If $|(V_i)_t| \leq \varepsilon n$, then
$\max_{j\in S}| V_j^{t/(r-1)}|\geq (1-\eps)n$.
\end{lemma}

\begin{proof}
Since any edge meeting $(V_i)_t$ meets $\cup_{j \neq i} V_{j} ^{t/(r-1)}$
and
$|V_j ^{+}| \leq (1-\varepsilon)n$ for $j \in B$, there must (see (ii)) be some
$j \in S$ with $|V_{j} ^{t/(r-1)}| \geq (1- \varepsilon)n$.
\end{proof}

We now claim $\psi (V_i) \leq 2r^2\psi (V)/\lambda$ for all $i$.
For $i \in S$, we do a little better: Lemma~\ref{small sets have no value} gives  $\psi (V_i ^{-}) \leq r\psi (V^{+})  / \lambda$, and balance (of $w$) then implies $\psi (V_i) =2\psi (V_i^{-}) \leq r\psi (V)  / \lambda$.
For $i\in B$ write $W$ for $V_i$ (just to avoid some double subscripts) and set
$T = \sup \{t \ : \ |W_t| \geq \varepsilon n\}$.  Then
\[
\psi (W^-) = \psi (W_T) + \psi (W^{-} \setminus W_T)
\leq \psi (W_T) + T|W^{-} \setminus W_T|.
\]
Since $|W_T| < \varepsilon n$, Lemma~\ref{small sets have no value} gives $\psi (W_T) \leq r\psi (V^{+})  / \lambda$.  On the other hand, $|W_t| \geq \varepsilon n$ for $t\in [0 ,T)$, with Lemma~\ref{large sets have no value}, implies that there is a $j\in S$ with $| V_j^{t/(r-1)}|\geq (1-\eps)n$ for all such $t$.  Thus
\begin{eqnarray*}\displaystyle
\mbox{$(1-\eps)T|W^{-} \setminus W_T|$} &\leq& \mbox{$(1-\eps)nT
~\leq~
\int_{0}^{T} | V_{j} ^{t/(r-1)} | \dd{t}  \leq \int_{0}^{\infty} | V_{j} ^{t/(r-1)} | \dd{t}$}
\\
&=& \mbox{$(r-1)\psi ( V_{j} ^{+}) ~ \leq ~ r^2 \psi (V^{+})/\lambda.$}
\end{eqnarray*}
So, combining, we have
$\psi (W) = 2\psi (W^{-}) \leq 2r^2\psi (V)  / \lambda $ (establishing the claim)
and
\[
\mbox{$\psi (V) = \sum_{i} \psi (V_i) \leq 2r^3 \psi (V)/\lambda.$}
\]
But since $2r^3 < \lambda$, this forces $\psi (V) = 0$ and so $w \equiv 0$.
\end{proof}

\begin{proof}[Proof of Theorem \ref{r-partiteHypergraph}]
Set $\gl = 4r^3$, $\eps = (2r\gl)^{-1}$ and $k =2r\eps^{-r}$ (so $k$ is a little more than $r^{4r}$).  We show that w.h.p.\ $\Aitch = \MPOut{n}{k}{r}$ is as in Proposition~\ref{r-partite-expansion}.  As earlier, let $B(X,Y)$ be the event that every edge meeting $X$ meets $Y$.

\mn

Suppose first that $T$ and $U$ are fixed with $|U_i| = \lambda |T| \leq \lambda \varepsilon n$.  Then
\[
\mbox{$\pr(B(T,U)) \leq \left[1 -\left(1-\frac{\lambda|T|}{n} \right)^{r-1}  \right]^{k|T|}
\leq \left( \frac{r\lambda|T|}{n} \right) ^{k|T|}.$}
\]
Summing over choices of $T$ and $U$ bounds the probability that $\Aitch$ violates the assumptions
of the proposition for some $T$ and $U$ as in (i) by
\begin{eqnarray*}
\mbox{$r \sum_{t=1} ^{\varepsilon n} \Cc{n}{t}
\Cc{n}{\lambda t} ^{r-1} \left(\frac{r \lambda t}{n} \right)^{kt}$}
&\leq &
\mbox{$r \sum_{t=1} ^{\varepsilon n} \left( \frac{en}{t}\right)^t
\left( \frac{en}{\lambda t}\right)^{\lambda t(r-1)} \left(\frac{r \lambda t}{n} \right)^{kt}$}\\
&\leq&
\mbox{$\sum_{t=1} ^{\varepsilon n} \left[ ( r \lambda t/n)^{k-r\lambda} \lambda (er)^{r\lambda} \right]^{t}
= o(1).$}
\end{eqnarray*}

\mn

Now say $T$ and $U$ are fixed with $|T| = \varepsilon n$ and $|U_i| = (1- \varepsilon)n$.  Then
\[
\pr(B(T,U)) \leq  (1 - \eps^{r-1}) ^{k|T|} \leq \exp \left[- k|T|\varepsilon^{r-1} \right] \leq \exp \left[- kn \varepsilon^{r} \right].
\]
So summing over possibilities for $(T,U)$ bounds the probability of a violation
with $T$ and $U$ as in (ii) by
\[
r 2^{nr} \exp \left[- kn \varepsilon^{r} \right] \leq \exp \left[n(r-k \varepsilon^r) \right] = o(1). \qedhere
\]
\end{proof}

\section{Proof of Theorem \ref{theorem Krivelevich Stopping}}\label{sectionStopping}

We now turn to our proof of Theorem~\ref{theorem Krivelevich Stopping}, for which
we work with the following standard device for handling the process $\{\h_t\}$.

\mn

Let $\xi_S$, $S\in \lcho{[n]}{r}$, be independent random variables, each uniform from $[0,1]$,
and for $\gl\in [0,1]$, let $G(\gl)$ be the $r$-graph on $[n]$ with edge set
$\Ee(\gl) = \{ S \ : \ \xi_{S} \leq \gl \}$.
Members of $\eee(\gl)$ will be called $\gl$-\emph{edges}.
Note that with probability one, $G(0)$ is empty, $G(1)$ is complete, and the $\xi_S$'s are distinct.

\mn

Provided the $\xi_S$'s are distinct, this defines the discrete process
$\{\h_t\}$ in the natural way, namely by adding edges $S$ in the order in which their associated
$\xi_S$'s appear in $[0,1]$.
We will work with the following quantities, where $\cc=\eps \log n$ for some small fixed
(positive) $\eps$
and $g$ is a suitably slow $\go(1)$.

\begin{itemize}
\item $\gL =\min\{\gl: \mbox{$G(\gl)$ has no isolated vertices}\}$;
\item $W_{\lambda} = \{v \in [n] \ : \ d_{G(\gl)} (v) \leq \cc \}$;
\item $\gs = \frac{\log n - g(n)}{\Cc{n-1}{r-1}}$ and $\beta = \frac{\log n + g(n)}{\Cc{n-1}{r-1}}$;
\item $N= \{v :
\exists e \in \Ee(\beta),\ v \in e,\ e \cap W_{\gs} \neq \emptyset\}$
\end{itemize}
(so $N$ is $W_{\gs}$ together with its $\Ee(\beta)$-neighbors).

\bn
\textbf{Preview.}
With the above framework, our assignment is to show that $G(\gL)$ has a perfect matching w.h.p..
Perhaps the nicest part of this---and the point of coupling the different $G(\gl)$'s---is that,
so long as $\gL\in [\gs,\gb]$, which we will show holds w.h.p.,
the desired assertion on $G(\gL)$ follows \emph{deterministically} from a few
properties ((b)-(d)) of Lemma~\ref{stoppingLemma}) involving $G(\gs)$, $G(\gb)$
\emph{or both}; so by showing that the latter properties hold w.h.p.\ we avoid the need for a union bound
to cover possibilities for $\gL$.
Production of the fractional matching is then similar to (though somewhat simpler than)
what happens in \cite{krivelevich}:
the relatively few vertices of $W_\gL$ (and some others) are covered by an (ordinary) matching,
and the hypergraph induced by what's left has the expansion needed for Proposition~\ref{mainLemma}.

\begin{lemma}\label{stoppingLemma}
With the above setup (for fixed r) and $Z= n(\log n)^{-1/r}$, w.h.p.
\begin{itemize}
\item[(a)] $\gL\in [\gs, \beta]$;
\item[(b)] $\ga(G(\gs))<Z$;
\item[(c)] no $\beta$-edge meets $W_{\gs}$ more than once
and no $u\not\in W_{\gs}$ lies in more than one $\gb$-edge meeting
$N \setminus \{u\}$;
\item[(d)] each $X\sub V\sm W_{\gs}$ of size at most Z is $r$-expansive in $ G(\gs)$.

\end{itemize}
\end{lemma}

\begin{proof}
For (a), note that
the expected number of isolated vertices in $G(\gl)$ is $h(\gl):= n(1-\gl)^{\C{n-1}{r-1}}$.
The upper bound (i.e.\ $\gL< \gb$ w.h.p.) then follows from $h(\gb)=o(1)$, and the lower bound
is given by Chebyshev's Inequality (applied to the number of isolated vertices).

\mn

For (b), we have
\begin{eqnarray*}
\pr (\ga(G(\gb)) \geq Z) &<& \Cc{n}{Z} (1- \beta)^{\Cc{Z}{r}}
~<~ \left(en/Z\right)^Z \exp \left[-\beta\Cc{Z}{r} \right]\\
&=& \exp \left[Z \log(en/Z)-(1-o(1))(n/r)\log n (Z/n)^r\right]\\
&=& \exp \left[Z \log(en/Z)-\gO(n)\right] = o(1).
\end{eqnarray*}

\mn
The proofs of (c) and (d) are similarly routine but take a little longer.
Aiming for (c), set $p=\pr(\gz\leq \cc)$,
where $\gz$ is binomial with parameters
$\C{n-2}{r-1}$ and $\gs$.  Since $\mu:=\E \gz \sim \log n$,
a standard large deviation estimate (e.g., \cite[Thm.\ 2.1]{JLR}) gives
\[
p < \exp[-\mu \vp (-(\mu-\cc)/\mu)]  < n^{-1+\gd},
\]
where $\vp (x) = (x+1)\log (x+1)-x$ for $x\geq -1$ and
$\gd\approx \eps\log (1/\eps)$.

Failure of the first assertion in (c) implies existence of $S\in \complete{n}{r}$
and (distinct) $u,v\in S$ with $S\in G(\gb)$ and $u,v\in W_{\gs}$.
The probability that this occurs for a given $S,u,v$ is less than $\gb p^2$
(the $p^2$ bounding the probability that each of $u,v$ lies in at most $\cc$ edges not
containing the other),
so the probability that the assertion fails is less than
\[
\Cc{n}{r}r^2\gb p^2 \sim nr(\log n) p^2=o(1).
\]

If the second part of (c) fails, then we must be able to find a $u \notin W_\sigma$ as well as one of the
following configurations, in which $x,y \in W_\sigma$, $S_i \in G(\beta)$, and $a,b \in [n]$ (and
vertices and edges within a configuration are distinct):
\begin{itemize}
\item[(i)] $x, S_1, S_2$ with $x,u \in S_1\cap S_2$;
\item[(ii)] $x, y, S_1, S_2$ with $x, u \in S_1$, $y, u \in S_2$;
\item[(iii)] $x, a, S_1, S_2, S_3$ with $x, u \in S_1$, $x,a \in S_2$, $u,a \in S_3$;
\item[(iv)] $x,y, a ,S_1, S_2, S_3$ with $x,u \in S_1$, $y,a \in S_2$, $u,a \in S_3$;
\item[(v)] $x, a, S_1, S_2, S_3$ with $x,a \in S_1$, $u,a \in S_2 \cap S_3$;
\item[(vi)] $x, a,b, S_1, S_2, S_3, S_4$ with $x,a \in S_1$, $x,b \in S_2$, $u,a \in S_3$, $u,b \in S_4$;
\item[(vii)] $x,y, a,b, S_1, S_2, S_3, S_4$ with $x,a \in S_1$, $y,b \in S_2$, $u,a \in S_3$, $u,b \in S_4$;
\item[(viii)] $x,a,b, S_1, S_2, S_3$ with $x,a,b \in S_1, u,a\in S_2, u,b\in S_3$.
\end{itemize}
Thus, with $M=\C{n-2}{r-2}$, summing probabilities for these possibilities bounds the probability of violating the second part of (c) by
\begin{eqnarray*}
n^2pM^2\gb^2 + n^3 p^2 M^2 \beta^2 + n^3 p M^3 \beta^3 + n^4 p^2 M^3 \beta^3 + n^3 p M^3 \beta^3
~~~~~\\
~~~~~~~~~~~~~ ~~~~~~~~~+ n^4 p M^4 \beta^4 + n^5 p^2 M^4 \beta^4 + n^4pM^2\Cc{n-3}{r-3}\gb^3= o(1).
\end{eqnarray*}

For (d) it is enough to bound (by $o(1)$) the probability that for some (nonempty) $X\sub V$
of size $x\leq Z$ and $Y\sub V\sm X$ of size $rx$,
\beq{XY}
\mbox{\emph{there are at least $\cc x/r$ $\gs$-edges meeting both $X$ and $Y$.}}
\enq
For given $X,Y$ the expected number of such edges is less than
\[
x \cdot rx \Cc{n-2}{r-2} \gs < x r^2 \tfrac{Z \log n}{n-1}  =: bx.
\]
(The first inequality is a significant giveaway for small $x$, but we have lots of room.)
So, again using \cite[Thm.\ 2.1]{JLR}, we find that the probability of \eqref{XY}
is less than
\[
\exp[-(\cc x/r)\log (\cc/(erb)] < \exp[-\gO(\cc x \log\log n)],
\]
while the number of possibilities for $(X,Y)$ is less than
\[
\Cc{n}{x}\Cc{n}{rx} < \exp[(r+1)x\log (n/x)] = \exp[O(x\log n)],
\]
and the desired $o(1)$ bound follows.
\end{proof}

\begin{proof}[Proof of Theorem~\ref{theorem Krivelevich Stopping}.]
By Lemma~\ref{stoppingLemma} it is enough to show that if (a)-(d) of the lemma hold then
$G(\gL)$ has a perfect fractional matching; so we assume we have these conditions and proceed
(working in $G(\gL)$).

According to (c) (and the definition of $\gL$), $G(\gL)$ admits a
\emph{matching}, $M$, covering $W_{\gs}$ (each edge of which contains exactly one vertex of $W_{\gs}$).
Let $W$ be the set of vertices covered by $M$ (so $W$ consists of $W_{\gs}$ plus some
subset of $N\sm W_{\gs}$), and $H= G(\gL)-W$ (as usual meaning that the edges of $H$
are the edges of $G(\gL)$ that miss $W$).
It is enough to show that $H$ has a perfect fractional matching, which will follow from
Proposition~\ref{mainLemma} if we show
\beq{(r-1)exp}
\mbox{each independent set $X$ of $H$ is $(r-1)$-expansive.}
\enq
\begin{proof}  Since such an $X$ is also independent in $G(\gs)$, (b) gives $|X|\leq Z$, and
(d) then says $X$ is $r$-expansive in $G(\gs)$, \emph{a fortiori} in $G(\gL)$.
On the other hand, since $X\cap W_{\gs}=\0$, (c) guarantees that the $\gb$-edges (so also the $\gL$-edges)
meeting $X$ and \emph{not} contained in $V(H)$ can be covered by some $U\sub W$ of size at most $|X|$
(namely, (c) says each $x\in X$ lies in at most one such edge).
It follows that the $\gL$-edges meeting $X$ that \emph{do} belong to $H$ cannot be
covered by $(r-1)|X|$ vertices of $V(H)\sm X$.
\end{proof}
\end{proof}

\bibliography{myPublications}
\end{document}